\documentclass[printscheme,screen]{aomart}

\usepackage{amssymb,mathtools}
\usepackage[shortlabels]{enumitem}
\usepackage{microtype}
\usepackage{needspace}
\usepackage{xurl}
\usepackage{booktabs,array}
\copyrightnote{}
\doinumber{}
\fancypagestyle{firstpage}{\fancyhf{}\fancyfoot[C]{\thepage}}

\numberwithin{equation}{section}
\newcommand{\R}{\mathbb R}
\newcommand{\N}{\mathbb N}
\newcommand{\Cone}{\operatorname{Cone}}
\newcommand{\BR}{\operatorname{BR}}
\newcommand{\cL}{\mathcal L}
\newcommand{\rec}{\operatorname{rec}}

\theoremstyle{plain}
\newtheorem{theorem}{Theorem}[section]
\newtheorem{proposition}[theorem]{Proposition}
\newtheorem{lemma}[theorem]{Lemma}
\newtheorem{corollary}[theorem]{Corollary}
\theoremstyle{definition}
\newtheorem{definition}[theorem]{Definition}
\theoremstyle{remark}
\newtheorem{remark}[theorem]{Remark}
\renewcommand{\theequation}{%
  \ifnum\value{section}=1
    \arabic{equation}%
  \else
    \thesection.\arabic{equation}%
  \fi}
\renewcommand{\thetheorem}{%
  \ifnum\value{section}=1
    \arabic{theorem}%
  \else
    \thesection.\arabic{theorem}%
  \fi}

\title[Bounded Ratios of Lorentzian Polynomials I]
{Bounded Ratios of Lorentzian Polynomials\texorpdfstring{\\}{ }
I: The Ternary Theory and\texorpdfstring{\\}{ }
Optimal Bounding Constants}

\author{Dijia Chen}
\address{Department of Mathematics, University of Wisconsin--Madison, Madison, WI, USA}
\email{dchen426@wisc.edu}
\author{Bowen Gan}
\address{Institute of Mathematical Sciences, ShanghaiTech University, Shanghai, China}
\email{ganbw2023@shanghaitech.edu.cn}
\author{Ivy Liu}
\address{Department of Mathematics, University of Wisconsin--Madison, Madison, WI, USA}
\email{ivy.liu@wisc.edu}
\author{Zemeng Wang}
\address{Department of Mathematics, University of Wisconsin--Madison, Madison, WI, USA}
\email{zwang3694@wisc.edu}
\author{Chengzhi Wu}
\address{Department of Mathematics, University of Wisconsin--Madison, Madison, WI, USA}
\email{cwu493@wisc.edu}

\hypersetup{colorlinks=true,linkcolor=blue,citecolor=blue,urlcolor=blue,
pdftitle={Bounded Ratios of Lorentzian Polynomials I: The Ternary Theory and Optimal Bounding Constants},
pdfauthor={Dijia Chen, Bowen Gan, Ivy Liu, Zemeng Wang, Chengzhi Wu}}

\begin{document}

\begin{abstract}
We study bounded coefficient ratios and their optimal bounding constants for ternary Lorentzian polynomials. For every degree $d\geq2$ and every fixed ternary $M$-convex support, we prove that bounded ratios are generated by those of quadratic derivatives. We express the optimal bounding constants through a variational formula. For full support, we determine all linear relations among the logarithms of triangular ratios in every degree. For ternary cubics, we compute sharp constants for products of any two distinct triangular ratios. Finally, we compare the bounds for Lorentzian polynomials with those for volume polynomials and rank-three matroid basis counts.
\end{abstract}

\maketitle
\thispagestyle{empty}
\raggedbottom
{\footnotesize\tableofcontents}

\section{Introduction and statements of the main theorems}

The \emph{Alexandrov--Fenchel} inequality gives multiplicative inequalities among \emph{mixed volumes}. For convex bodies $K_1,\ldots,K_n\subset\mathbb{R}^d$, the polynomial $\operatorname{Vol}(x_1K_1+\cdots+x_nK_n)$ records their mixed volumes in its coefficients. Fixing $d-2$ of the mixed-volume arguments gives a symmetric matrix $P=(p_{ij})$ satisfying $p_{ii}p_{jj}\le p_{ij}^2$ \cite[Chapter~7]{Schneider}. Br\"and\'en and Huh \cite{BH} introduced Lorentzian polynomials as a common framework for these and other examples from geometry and combinatorics. We study multiplicative coefficient inequalities for the class of Lorentzian polynomials.

Huang, Huh, Soskin, and Wang call $\ \ \displaystyle \prod_{i\le j}p_{ij}^{\alpha_{ij}}\ $ a \emph{bounded ratio} if, for fixed real exponents $\alpha_{ij}$, it is uniformly bounded above as $P=(p_{ij})$ ranges over \emph{Lorentzian matrices} with positive entries \cite[Definition~1.1]{HHSW}. They characterized these ratios by identifying the cone of \emph{reduced bounded ratios} with the dual of the \emph{cut cone} \cite[Theorem~B]{HHSW}. The generators in small sizes are particularly explicit. In sizes three and four, every bounded ratio is a product of nonnegative real powers of the \emph{triangular ratios}
\[
 \frac{p_{ii}p_{jk}}{p_{ij}p_{ik}},
\]
where $i,j,k \text{ are pairwise distinct}$. Each triangular ratio is bounded above by $2$. For matrices of size five, one also needs the \emph{pentagonal ratios}. They are obtained by permuting the indices in
\[
 \frac{p_{12}p_{13}p_{23}p_{44}p_{45}p_{55}}
      {p_{14}p_{15}p_{24}p_{25}p_{34}p_{35}},
\]
and each is bounded above by $4$. Together with the triangular ratios, they generate every bounded ratio in size five \cite[Example~1.4 and Theorem~A]{HHSW}.

In three variables, the bounded-ratio cones of quadratic derivatives have explicit generators, and the possible logarithms of these ratios form convex sets. This allows us to determine both bounded ratios and their optimal constants in higher degrees. In Part~II, we study the corresponding problem in general.

Every nonzero quadratic derivative of a Lorentzian polynomial is Lorentzian \cite[Corollary~2.11]{BH}. Its normalized coefficients come from the original polynomial with shifted indices, so any bounded ratio for the derivative gives a bounded ratio for the original polynomial. In three variables, we prove that these ratios generate all bounded ratios and determine their optimal bounding constants.

To state our results, write $[3]=\{1,2,3\}$ and $\Delta_3^m=\{\alpha\in\mathbb Z_{\geq0}^3:|\alpha|=m\}$ for $m\geq0$, where $|\alpha|=\alpha_1+\alpha_2+\alpha_3$. Let $e_1,e_2,e_3$ be the standard basis of $\R^3$. Fix $d\geq2$ and a nonempty $M$-convex support $S\subseteq\Delta_3^d$. With $\alpha!=\alpha_1!\alpha_2!\alpha_3!$, use the factorial normalization
$$
 f(\mathbf x)=\sum_{\alpha\in S}
       \frac{c_\alpha}{\alpha!}\mathbf x^\alpha,
 \qquad c_\alpha>0,
$$
and let $\cL(S)$ denote the Lorentzian polynomials with support exactly $S$. For $w\in\R^S$, put $R_w(f)=\prod_{\alpha\in S}c_\alpha^{w_\alpha},\quad B_S(w)=\sup_{f\in\cL(S)}R_w(f)$. The cone of exponent vectors $w$ with $B_S(w)<\infty$ is denoted by $\BR(S)$. Let $K(S)$ be the cone generated by bounded ratios obtained from nonzero quadratic derivatives, viewed as ratios of normalized coefficients of $f$. The definitions of $M$-convex sets and Lorentzian polynomials are recalled in Section~\ref{sec:preliminaries}. We call $S=\Delta_3^d$ \emph{full support}.

Our first result determines the cone of bounded ratios in three variables. The bounded ratios obtained from quadratic derivatives generate all bounded ratios, even when the support is not full. Thus boundedness in every degree reduces to the quadratic support cases classified in Section~\ref{sec:the-bounded-ratio}.
\begin{theorem}
\label{thm:intro-generation}
For every nonempty $M$-convex $S\subseteq\Delta_3^d$ with $d\geq2$, one has $\BR(S)=K(S)$.
\end{theorem}

This is Theorem~\ref{thm:local-global} below. For full support, the cone is generated by triangular ratios, and Theorem~\ref{thm:extreme-rays} shows that every such generator spans an exposed extreme ray. Thus this generating set is minimal.

To formulate the optimal-constant problem, write $S_\beta$ for the support of the quadratic derivative $\partial^\beta f$, $\beta\in\Delta_3^{d-2}$. For each nonzero derivative, take the generators of its bounded-ratio cone given in Corollary~\ref{cor:local-cones}. Let $M_S$ be the matrix whose rows are their exponent vectors, shifted to the coefficient indices of $S$ and grouped by $\beta$. A derivative with no generators contributes no row.

For each quadratic support $T$, let $h_T$ be the support function of the set of logarithms of its generator ratios, with $h_T(0)=0$ when there are no generators. For a nonnegative weight vector $\lambda$, let $\lambda_\beta$ denote its weights in the $\beta$ group and set $H_S(\lambda)=\log B_S(M_S^{\mathsf T}\lambda)$. The detailed construction is given in Section~\ref{sec:optimal-bounding-constants}.

Multiplying sharp quadratic bounds gives an upper bound for the original coefficient ratio. This bound can depend on the choice of weights, even when the exponent vector is fixed. The following theorem gives the optimal constant by varying the weights without changing that vector.
\begin{theorem}
\label{thm:intro-constants}
Suppose that $S\subseteq\Delta_3^d$ is nonempty and $M$-convex, with $d\geq2$. For every nonnegative vector $\lambda$ with one weight for each row of $M_S$, one has
$$
 H_S(\lambda)=\inf_\mu\sum_{\beta\in\Delta_3^{d-2}}
 h_{S_\beta}(\mu_\beta),
$$
where the infimum is over nonnegative vectors $\mu$ satisfying $M_S^{\mathsf T}\mu=M_S^{\mathsf T}\lambda$.
\end{theorem}

This is Theorem~\ref{thm:general-optimal} below.

For full support, each quadratic derivative contributes three triangular ratios. For $\{i,j,k\}=\{1,2,3\}$, we write
$$
 T_{\beta,i}=T_\beta(j,k,i)
 =\frac{c_{\beta+2e_i}c_{\beta+e_j+e_k}}
        {c_{\beta+e_i+e_j}c_{\beta+e_i+e_k}}.
$$
We denote its exponent vector by $t_{\beta,i}$, and write $M_d=M_{\Delta_3^d}$. Within each $\beta$ group, we order the rows of $M_d$ as $t_{\beta,1},t_{\beta,2},t_{\beta,3}$. We put $L_d=\operatorname{im}M_d$, the space of logarithms of these ratios obtained by varying all positive coefficients.

The following theorem describes the relations imposed by shared coefficients. It also shows that none of the displayed equations follows from the others. Their coefficients give all changes of weights that leave a coefficient ratio unchanged.
\begin{theorem}
\label{thm:intro-compatibility}
Suppose that $d\geq3$. A vector of logarithmic ratios $r=(r_{\beta,i})_{\beta\in\Delta_3^{d-2},\,i\in[3]}$ belongs to $L_d$ if and only if the following relations hold for every $\gamma\in\Delta_3^{d-3}$:
$$
\begin{aligned}
 r_{\gamma+e_3,1}+r_{\gamma+e_1,3}
 &=r_{\gamma+e_2,1}+r_{\gamma+e_2,3},\\
 r_{\gamma+e_3,2}+r_{\gamma+e_2,3}
 &=r_{\gamma+e_1,2}+r_{\gamma+e_1,3}.
\end{aligned}
$$
These $(d-1)(d-2)$ relations are linearly independent.
\end{theorem}

This is Theorem~\ref{thm:explicit-compatibility} below. Its proof is algebraic: we determine $\ker M_d$ and construct explicit polynomial relations among the rows of $M_d$. Convexity is used separately in the proof of Theorem~\ref{thm:general-optimal}. For cubics, the compatibility relations lead to sharp bounds for products of two distinct triangular ratios (Theorem~\ref{thm:cubic-two-generator}).

Finally, we apply these inequalities to volume polynomials and matroid basis counts. In some cases the constants remain sharp, while in others the additional structure gives stronger inequalities.

\noindent\textbf{Related work.}\enspace Independently and around the same time, Bathija, Rohatgi, and Soskin~\cite{BRS} studied bounded ratios for Lorentzian polynomials. For full support, their results overlap with our description of the bounded-ratio cone, the exposed extreme rays, the formula for optimal constants, and some sharp cubic bounds. Their work also considers arbitrary numbers of variables. The present paper treats every fixed ternary $M$-convex support, determines the relations among triangular ratios in every degree for full support, and gives applications to volume polynomials and rank-three matroids.

After this manuscript was completed, Baldi and Kummer~\cite[Theorem~1.1 and Corollary~1.2]{BaldiKummer2026} posted a related preprint. They characterize bounded ratios on sets defined by polynomial equations and inequalities. Applied to Lorentzian polynomials, their result gives another proof of Theorem~\ref{thm:local-global} for every fixed $M$-convex support.

The rest of the paper is organized as follows. Section~\ref{sec:preliminaries} fixes the notation and recalls the Hessian characterization of Lorentzian polynomials. Section~\ref{sec:the-bounded-ratio} determines the bounded-ratio cone and its geometry. Section~\ref{sec:optimal-bounding-constants} derives the formula for sharp constants, and Section~\ref{sec:compatibility-constraints} gives the compatibility relations for full support. Section~\ref{sec:cubic} treats cubics, and Section~\ref{sec:applications} compares the geometric and combinatorial bounds.

\noindent\textbf{Acknowledgments.} This project grew out of a reading group on Lorentzian polynomials organized by Botong Wang. We thank him for introducing us to the subject and for his guidance and encouragement. We also thank Daniel Soskin for his helpful suggestions and observations.

\section{Preliminaries}
\label{sec:preliminaries}

\subsection{Lorentzian polynomials}

Throughout the paper, we take $n\geq1$. We adopt the notation of \cite[Section~2.1]{BH} and denote the set of nonnegative integers by $\N$, and let $[n]=\{1,\ldots,n\}$. We define the \textit{d-th discrete simplex} $\Delta_n^d=\{\alpha\in\N^n:|\alpha|=d\}$ for $d\geq0$, where $|\alpha|=\sum_i\alpha_i$. We let $\{e_1,\ldots,e_n\}$ denote the standard basis of $\R^n$. For a multi-index $\alpha$, we set $\alpha!=\prod_i\alpha_i!$, $\mathbf x^\alpha=\prod_i x_i^{\alpha_i}$, and $\partial^\alpha=\prod_i\partial_i^{\alpha_i}$, where $\partial_i=\partial/\partial x_i$.

We use the factorial normalization of \cite[Section~3.3]{BH}. A nonzero homogeneous polynomial $f$ of degree $d\geq1$ with nonnegative real coefficients has the expansion
\begin{equation*}
 f(\mathbf x)=\sum_{\alpha\in\Delta_n^d}
                  \frac{c_\alpha}{\alpha!}\,\mathbf x^\alpha,
\end{equation*}
where each $c_\alpha$ is nonnegative. The \emph{Hessian} of $f$ is the symmetric matrix
\[
\mathcal{H}_f(\mathbf x)
=
\bigl(\partial_i\partial_j f\bigr)_{i,j=1}^{n}.
\] 
The support of $f$ is $S=\operatorname{supp}(f)=\{\alpha\in\Delta_n^d:c_\alpha>0\}$.

Br\"and\'en and Huh \cite{BH} observed that the supports of stable polynomials of degree 2 share a combinatorial property known as \textit{M-convexity}. They used the symmetric exchange formulation \cite[Section~4.1]{Murota} in the arguments about supports as below.
\begin{definition}
A nonempty set $S\subseteq\Delta_n^d$ is \emph{$M$-convex} if, for every $\alpha,\beta\in S$ and every $i$ with $\alpha_i>\beta_i$, there is an index $j$ with $\alpha_j<\beta_j$ such that both $\alpha-e_i+e_j$ and $\beta+e_i-e_j$ belong to $S$.
\end{definition}

Under this normalization, differentiation shifts the coefficient indices. In other words, for $d\geq2$ and $\beta\in\Delta_n^{d-2}$, the quadratic derivative has the expansion
\begin{equation}\label{eq:quadratic-derivative-expansion}
  \partial^\beta f(\mathbf x)=
 \sum_{\eta\in\Delta_n^2}\frac {c_{\beta+\eta}}{\eta!}\,\mathbf x^\eta.   
\end{equation}

Its support $S_\beta$ and Hessian matrix $\mathcal{H}_\beta(f)$ are
\[
 S_\beta=\{\eta\in\Delta_n^2:\beta+\eta\in S\},\qquad
 \mathcal{H}_\beta(f)=\bigl(c_{\beta+e_i+e_j}\bigr)_{i,j\in[n]}.
\]

In a quadratic support, the points $2e_i$ correspond to diagonal entries of the Hessian, and the points $e_i+e_j$ with $i\ne j$ correspond to off-diagonal entries. We refer to them as diagonal and off-diagonal points, respectively. If $S$ is $M$-convex, then every nonempty $S_\beta$ is $M$-convex. Indeed, for $\eta,\xi\in S_\beta$, symmetric exchange applied to $\beta+\eta$ and $\beta+\xi$ gives the required points of $S_\beta$ after subtracting $\beta$.

\Needspace{14\baselineskip}
Br\"and\'en and Huh \cite[Definition~2.6 and Theorem~2.25]{BH} showed that Lorentzian polynomials can be described in terms of M-convex supports and the Hessians of their quadratic derivatives. We use these two conditions in the following definition.

\begin{definition}\label{def:LP}
A nonzero homogeneous polynomial $f$ of degree $d\geq2$ with nonnegative coefficients is \emph{Lorentzian} if its support is $M$-convex and, for every $\beta\in\Delta_n^{d-2}$, the matrix $\mathcal{H}_\beta(f)$ has at most one positive eigenvalue. The zero polynomial and every homogeneous linear polynomial with nonnegative coefficients are also regarded as Lorentzian.
\end{definition}

For a nonempty $M$-convex $S\subseteq\Delta_n^d$ with $d\geq1$, we denote by $\cL(S)$ the class of Lorentzian polynomials with support exactly $S$ and put $\cL(n,d)=\cL(\Delta_n^d)$. For a ternary polynomial of degree at least two with full support, \emph{strictly Lorentzian} means that every quadratic Hessian has one positive and two negative eigenvalues.

Following \cite[Section~1]{HHSW}, we call a symmetric matrix with nonnegative entries \emph{Lorentzian} if it has at most one positive eigenvalue.

\subsection{Bounded ratios and quadratic derivatives}

We follow the terminology of \cite[Definition~1.1]{HHSW} and \cite[Definitions~2.2--2.3]{BRS}.
\begin{definition}
Suppose that $S\subseteq\Delta_n^d$ is nonempty and $M$-convex, with $d\geq1$. For $f\in\cL(S)$ and $w=(w_\alpha)_{\alpha\in S}\in\R^S$, we set $R_w(f)=\prod_{\alpha\in S}c_\alpha^{w_\alpha}$. We call $R_w$, or its exponent vector $w$, a \emph{bounded ratio} on $\cL(S)$ if there is a constant $C>0$ such that $R_w(f)\leq C$ for every $f\in\cL(S)$. The set of all such vectors is denoted by $\BR(S)$. For full support, we write $\BR(n,d)=\BR(\Delta_n^d)$.
\end{definition}

Since $R_{au+bv}=R_u^aR_v^b$ for $a,b\geq0$, the set $\BR(S)$ is a convex cone. We denote the standard coordinate vectors in $\R^S$ by $\varepsilon_\alpha$, $\alpha\in S$.

We fix $d\geq2$ and $\beta\in\Delta_n^{d-2}$ with $S_\beta\ne\varnothing$. For $v\in\R^{S_\beta}$, we define $\iota_\beta(v)\in\R^S$ by placing $v_\eta$ at the index $\beta+\eta$ and putting zero at every other index of $S$. The normalization gives $R_{\iota_\beta(v)}(f)=R_v(\partial^\beta f)$.

The shifted bounded ratios from quadratic derivatives generate the following cone.
\begin{definition}
For a nonempty $M$-convex $S\subseteq\Delta_n^d$ with $d\geq2$, we define $K(S)\subseteq\R^S$ to consist of all finite sums of vectors $\iota_\beta(v)$ with $\beta\in\Delta_n^{d-2}$, $S_\beta\ne\varnothing$, and $v\in\BR(S_\beta)$.
\end{definition}

Since each nonzero quadratic derivative of a Lorentzian polynomial is Lorentzian, one has
\begin{equation}\label{eq:local-contained}
 K(S)\subseteq\BR(S).
\end{equation}

Following \cite[Definition~1.1]{HHSW}, we define the optimal bounding constant by $B_S(w)=\sup_{f\in\cL(S)}R_w(f)$. Thus $w\in\BR(S)$ exactly when $B_S(w)<\infty$. For the empty quadratic support, we set $\BR(\emptyset)=\{0\}$. We write $\Cone(G)$ for the cone generated by $G$, with $\Cone(\varnothing)=\{0\}$.

For a convex cone $C\subseteq\R^S$, its \emph{polar cone} $C^\circ$ consists of the vectors $q$ such that $\langle w,q\rangle\leq0$ for every $w\in C$ \cite[Section~14]{Rockafellar}.

From now on, we work in three variables. For $\beta\in\Delta_3^{d-2}$ and $\{i,j,k\}=\{1,2,3\}$, the triangular ratios of \cite[Example~1.4]{HHSW}, applied to quadratic derivatives as in \cite[Definition~3.19]{BRS}, are
$$
\begin{aligned}
 T_{\beta,i}=T_\beta(j,k,i)
   &=\frac{c_{\beta+2e_i}c_{\beta+e_j+e_k}}
           {c_{\beta+e_i+e_j}c_{\beta+e_i+e_k}},\\
 t_{\beta,i}=t_\beta(j,k,i)
   &=\varepsilon_{\beta+2e_i}+\varepsilon_{\beta+e_j+e_k}
     -\varepsilon_{\beta+e_i+e_j}-\varepsilon_{\beta+e_i+e_k},
\end{aligned}
$$
where $t_{\beta,i}$ is the exponent vector of $T_{\beta,i}$. We use these expressions only when all four indices belong to $S$. The index $i$ specifies the diagonal Hessian entry in the ratio. The conditions in Section~\ref{sec:the-bounded-ratio} give $T_{\beta,i}\leq2$.

\section{The bounded-ratio cone}\label{sec:the-bounded-ratio}

\subsection{Quadratic supports}

For quadratic supports, we abbreviate $ii=2e_i$ and $ij=e_i+e_j$ for $i\ne j$. We identify supports that differ by a permutation of the three variables. The empty support represents the zero quadratic. The description of quadratic supports in \cite[Lemma~3.11]{BH} gives the following list in three variables.

\begin{lemma}
\label{lem:quadratic-supports}
Up to permutation, the empty set and the nonempty $M$-convex subsets of $\Delta_3^2$ are
$$
\begin{gathered}
\emptyset,\quad \{11\},\quad \{12\},\quad \{11,12\},\quad
\{12,13\},\quad \{11,12,13\},\\
\{11,12,22\},\quad \{12,13,23\},\quad
\{11,12,13,23\},\quad
\{11,12,13,22,23\},\quad \Delta_3^2.
\end{gathered}
$$
\end{lemma}

\begin{proof}
Symmetric exchange requires $ij$ when $ii,jj$ belong to the support. It also requires $ij,ik$ when $ii,jk$ belong to the support and the three indices are distinct. We classify the possibilities by the diagonal points. With no diagonal points, any subset of $\{12,13,23\}$ satisfies exchange. If $11$ is the only diagonal point, then $23$ can occur only when both $12,13$ occur. If the diagonal points are $11,22$, then $12$ must occur, and $13,23$ occur together. Three diagonal points give full support. These cases give the list, and each listed support satisfies exchange.
\end{proof}

We consider a symmetric matrix $P=(p_{ij})$ that is positive on a quadratic support $T$ and zero elsewhere. We identify $P$ with the quadratic whose normalized coefficients are its entries and use the same ratio notation $R_w(P)$.

A consequence of Cauchy interlacing \cite[Section~4.3]{HornJohnson} reduces the eigenvalue condition to a determinant test when a principal minor of order two is negative.
\begin{lemma}
\label{lem:interlacing-det-test}
If a real symmetric $3\times3$ matrix has a negative $2\times2$ principal minor, it has at most one positive eigenvalue exactly when its determinant is nonnegative.
\end{lemma}

\begin{proof}
With eigenvalues $\lambda_1,\lambda_2,\lambda_3$ listed in nonincreasing order, interlacing gives $\lambda_1>0>\lambda_3$. Thus the determinant is nonnegative if and only if $\lambda_2\leq0$.
\end{proof}

We use the following ratios from \cite[Examples~1.3--1.4]{HHSW}:
$$
 U=\frac{p_{11}p_{22}}{p_{12}^2},\qquad
 X=\frac{p_{11}p_{23}}{p_{12}p_{13}},\qquad
 Y=\frac{p_{22}p_{13}}{p_{12}p_{23}},\qquad
 Z=\frac{p_{33}p_{12}}{p_{13}p_{23}}.
$$
We use a ratio only when all entries appearing in it are positive, and we denote its exponent vector by the corresponding lower-case letter. We express the determinant using $F(X,Y,Z)=XYZ+2-X-Y-Z$. The full-support case follows the criterion in \cite[Section~4]{HHSW}. The other cases describe the proper quadratic supports.

\begin{lemma}
\label{lem:local-lorentzian-tests}
Let $P$ be a quadratic polynomial with positive coefficients on a support $T$ listed in Lemma~\ref{lem:quadratic-supports}. Up to a permutation of the variables, $P$ is Lorentzian exactly under the following conditions.
\begin{enumerate}[label=\textup{(\roman*)}]
\item If $T$ is not one of the supports in \textup{(ii)--(v)}, every choice of positive coefficients on $T$ gives a Lorentzian polynomial.
\item For $T=\{11,12,22\}$, the condition is $U\leq1$.
\item For $T=\{11,12,13,23\}$, the condition is $X\leq2$.
\item For $T=\{11,12,13,22,23\}$, the condition is $X+Y\leq2$.
\item For $T=\Delta_3^2$, the conditions are $XY,XZ,YZ\leq1,\quad F(X,Y,Z)\geq0$. These conditions also imply $X+Y,X+Z,Y+Z\leq2$.
\end{enumerate}
\end{lemma}

\begin{proof}
Case (i) follows from the preceding determinant test, except for the empty support and $\{11\}$, which are immediate. Case (ii) is the $2\times2$ determinant condition. In the remaining cases, one has $\det P=p_{12}p_{13}p_{23}F(X,Y,Z)$. In this identity, we set a ratio equal to zero if its numerator contains a missing diagonal entry. The determinant test gives (iii) and (iv), since these supports have a negative principal minor. For full support, the inequalities $XY,XZ,YZ\leq1$ are equivalent to $p_{ii}p_{jj}\leq p_{ij}^2$ for distinct $i,j$. If one of these inequalities is strict, the determinant test applies. If all three hold with equality, positivity forces $P$ to have rank one. This proves (v). Finally, $F=(XY-1)Z+2-X-Y\geq0$ and $XY\leq1$ give $X+Y\leq2$; the other two sum inequalities follow by symmetry.
\end{proof}

\begin{corollary}
\label{cor:unit-local-ratios}
A quadratic with positive coefficients on a support in Lemma~\ref{lem:quadratic-supports} is Lorentzian if the ratios in the corresponding list below are at most $1$: $\varnothing,\quad \{U\},\quad \{X\},\quad \{X,Y\},\quad \{X,Y,Z\}$, for cases \textup{(i)--(v)} of Lemma~\ref{lem:local-lorentzian-tests}, respectively.
\end{corollary}

\begin{proof}
Cases \textup{(i)--(iv)} follow directly from the local tests. In case \textup{(v)}, $X,Y,Z\leq1$ imply $XY,XZ,YZ\leq1$ and $F(X,Y,Z)=(1-X)(1-Y)+(1-Z)(1-XY)\geq0$. The conditions in case~(v) therefore hold.
\end{proof}

These tests also determine the bounded-ratio cone on each quadratic support. The full-support case was determined in \cite[Theorem~B and Example~1.4]{HHSW}.
\begin{corollary}
\label{cor:local-cones}
In cases \textup{(i)--(v)} of Lemma~\ref{lem:local-lorentzian-tests}, the cones $\BR(T)$ are, respectively,
$$
 \{0\},\quad \Cone\{u\},\quad \Cone\{x\},\quad
 \Cone\{x,y\},\quad \Cone\{x,y,z\}.
$$
\end{corollary}

\begin{proof}
In case \textup{(i)}, all coefficients on $T$ can vary independently over the positive real numbers, so only the zero exponent is bounded. The conditions in Lemma~\ref{lem:local-lorentzian-tests} bound all listed generator ratios.

For the reverse inclusion, we pass to logarithms as in \cite[Lemma~3.2]{HHSW}. If $w$ lies outside the listed cone, cone separation \cite[Section~11]{Rockafellar} gives $q$ with $\langle g,q\rangle\leq0$ for each listed generator $g$, but $\langle w,q\rangle>0$. For $t\geq0$, define $P(t)$ to have normalized coefficients $p_\alpha(t)=e^{tq_\alpha}$ at indices in $T$ and zero elsewhere. Every listed generator ratio is then at most $1$, so Corollary~\ref{cor:unit-local-ratios} shows that $P(t)$ is Lorentzian. Since $R_w(P(t))=e^{t\langle w,q\rangle}\to\infty$, the vector $w$ is not a bounded ratio.
\end{proof}

\subsection{The global cone}

We fix a nonempty $M$-convex support $S\subseteq\Delta_3^d$ with $d\geq2$. For each nonempty quadratic support $S_\beta$, we extend the generators in Corollary~\ref{cor:local-cones} by $\iota_\beta$. These finitely many vectors generate $K(S)$ by definition. The class $\cL(S)$ is nonempty by \cite[Theorem~3.10]{BH}.

\begin{theorem}
\label{thm:local-global}
For every nonempty $M$-convex $S\subseteq\Delta_3^d$ with $d\geq2$, one has $\BR(S)=K(S)$.
\end{theorem}

\begin{proof}
The inclusion $K(S)\subseteq\BR(S)$ is \eqref{eq:local-contained}. As in \cite[proof of Theorem~2.30]{BRS}, if $w\notin K(S)$, finite generation and cone separation \cite[Sections~11 and~19]{Rockafellar} give $q\in\R^S$ satisfying $\langle g,q\rangle\leq0,\quad \langle w,q\rangle>0$, where the first inequality holds for every shifted generator of $K(S)$. For $t\geq0$, we define $f_t$ to have normalized coefficients $c_\alpha(t)=e^{tq_\alpha}$. Every generator ratio from a quadratic derivative is at most $1$, so Corollary~\ref{cor:unit-local-ratios} shows that every nonzero quadratic derivative is Lorentzian. Hence $f_t\in\cL(S)$, whereas $R_w(f_t)=e^{t\langle w,q\rangle}$ tends to infinity as $t$ tends to infinity.
\end{proof}

For full support, each quadratic derivative contributes the three triangular generators $t_{\beta,i}$. This gives the following description.
\begin{corollary}
\label{cor:full-support}
For every $d\geq2$, one has
$$
 \BR(3,d)=\Cone\{t_{\beta,i}:
       \beta\in\Delta_3^{d-2},\ i\in[3]\}.
$$
\end{corollary}

\begin{proof}
For full support, Theorem~\ref{thm:local-global} and Corollary~\ref{cor:local-cones} give the stated generators.
\end{proof}

\begin{remark}[Soskin's observation]
\label{rem:support-extension}
Let $w\in\R^S$, and let $\widetilde w$ be its extension by zero to $\Delta_3^d$. Then $B_S(w)=B_{\Delta_3^d}(\widetilde w)$.

Indeed, if $G\in\cL(3,d)$, keeping only the monomials indexed by $S$ gives a polynomial in $\cL(S)$. To see this, note that $S$ is $M$-convex and that the quadratic conditions in Lemma~\ref{lem:local-lorentzian-tests} remain satisfied after the other monomials are removed. Up to permutation, the conditions for proper supports are either automatic, $U\leq1$, $X\leq2$, or $X+Y\leq2$, all of which follow from the full-support conditions in that lemma. Hence $B_{\Delta_3^d}(\widetilde w)\leq B_S(w)$.

Conversely, for $f\in\cL(S)$ and $\varepsilon>0$, replacing each $x_i$ by $x_i+\varepsilon(x_1+x_2+x_3)$ produces a full-support Lorentzian polynomial converging coefficientwise to $f$ as $\varepsilon$ tends to zero \cite[Theorem~2.10]{BH}. Continuity of the ratio gives the reverse inequality.
\end{remark}

We denote by $C\subseteq\R^3$ the set of triples $(\log X,\log Y,\log Z)$ arising from Lorentzian matrices with positive entries. Convexity of the set of logarithms of these ratios for strictly Lorentzian quadratics was proved in \cite[Lemma~4.4]{BRS}.

\begin{lemma}
\label{lem:row-convex}
The set $C$ is closed and convex.
\end{lemma}

\begin{proof}
For every positive triple $(X,Y,Z)$, the matrix with these diagonal entries and all off-diagonal entries equal to $1$ has ratios $X,Y,Z$. The local tests therefore describe $C$ by non-strict inequalities in continuous functions of the logarithms, so $C$ is closed. By \cite[Theorem~2.25]{BH}, every Lorentzian quadratic with positive coefficients is a coefficientwise limit of strictly Lorentzian quadratics. The ratios depend continuously on the positive coefficients, so $C$ is the closure of the corresponding set for strictly Lorentzian quadratics. Its convexity follows from \cite[Lemma~4.4]{BRS}.
\end{proof}

We denote by $A(S)$ the set of vectors $(\log c_\alpha)_{\alpha\in S}$ arising from polynomials in $\cL(S)$. The same convexity holds for coefficient logarithms on every fixed ternary $M$-convex support.
\begin{theorem}
\label{thm:log-convexity}
Suppose that $S\subseteq\Delta_3^d$ is nonempty and $M$-convex, with $d\geq2$. If $f,h\in\cL(S)$ have normalized coefficients $c_\alpha$ and $d_\alpha$, then, for every $0\leq t\leq1$, the polynomial with normalized coefficients $c_\alpha^t d_\alpha^{1-t}$ belongs to $\cL(S)$. Equivalently, $A(S)$ is convex.
\end{theorem}

\begin{proof}
The support remains $S$, so it suffices to check the quadratic derivatives, as in \cite[Proposition~3.25]{BH}. For a proper quadratic support, the tests either impose no condition, give a linear inequality in the logarithm of a ratio, or require $X+Y\leq2$. The set of pairs $(\log X,\log Y)$ satisfying this last inequality is convex. For full support, the logarithms of the ratios must lie in $C$. Each logarithm is a linear function of the coefficient logarithms. Thus each quadratic condition defines a convex set of coefficient logarithms, and their intersection is $A(S)$.
\end{proof}

For a nonempty convex set $D$, its \emph{recession cone} $\rec(D)$ consists of the directions $q$ for which $a+tq\in D$ for all $a\in D$ and $t\geq0$ \cite[Section~8]{Rockafellar}. For $A(S)$, these directions are determined by the generators of $K(S)$.

The bipolar theorem \cite[Section~14]{Rockafellar} then relates these directions to bounded ratios.
\begin{theorem}
\label{thm:recession-duality}
For every nonempty $M$-convex $S\subseteq\Delta_3^d$ with $d\geq2$, one has
$$
 \rec(A(S))=K(S)^\circ,\qquad
 \BR(S)=\rec(A(S))^\circ.
$$
\end{theorem}

\begin{proof}
Suppose that $q\in\R^S$ satisfies $\langle g,q\rangle\leq0$ for every shifted generator of $K(S)$. For any $a\in A(S)$ and $t\geq0$, the generator ratios do not increase along $a+tq$. The tests for proper quadratic supports remain valid. For full support, the inequalities $XY,XZ,YZ\leq1$ remain valid. On this region, the partial derivative of $F$ with respect to $X$ is $YZ-1\leq0$, and similarly for $Y$ and $Z$. Decreasing the ratios therefore cannot decrease $F$. All quadratic conditions remain valid, so $q\in\rec(A(S))$.

Conversely, if $q\in\rec(A(S))$, each generator ratio is bounded along $a+tq$. This forces $\langle g,q\rangle\leq0$ for every shifted generator of $K(S)$. Thus $\rec(A(S))=K(S)^\circ$. The remaining identity follows from Theorem~\ref{thm:local-global} and the bipolar theorem \cite[Section~14]{Rockafellar}.
\end{proof}

A ray of a convex cone is extreme if every decomposition of a nonzero vector on the ray as a sum of two vectors in the cone has both summands on the ray. It is exposed if a linear functional is nonpositive on the cone and vanishes there precisely on the ray. We give explicit exposing functionals \cite[Section~18]{Rockafellar} using the quadratic function in \cite[Lemma~5.2]{BRS}.

\begin{theorem}
\label{thm:extreme-rays}
For every $d\geq2$, $\beta\in\Delta_3^{d-2}$, and $i\in[3]$, the ray $\R_{\geq0}t_{\beta,i}$ is an exposed extreme ray of $\BR(3,d)$. These are all its extreme rays, and their number is $3\binom d2$. Thus the generating set in Corollary~\ref{cor:full-support} is minimal.
\end{theorem}

\begin{proof}
We set $q_\alpha=-\|\alpha\|_2^2$. Every generator $g$ has Euclidean norm $2$ and $\langle g,q\rangle=-2$. The two indices at which $t_{\beta,i}$ is positive differ by $\pm(2e_i-e_j-e_k)$, so they determine $i$ and then $\beta$, and the generators are distinct. For a fixed generator $g_0$, one has $\langle g,q+\tfrac12g_0\rangle =-2+\tfrac12\langle g,g_0\rangle$. The Cauchy--Schwarz inequality \cite[Section~0.6]{HornJohnson} shows that this expression is zero for $g=g_0$ and negative for every other generator. Thus the functional exposes the ray of $g_0$. There are $3|\Delta_3^{d-2}|=3\binom d2$ generators, and they generate the cone.
\end{proof}

\section{Optimal bounding constants}\label{sec:optimal-bounding-constants}

Quadratic derivatives share coefficients, so multiplying their separate optimal bounds need not give a sharp bound. We obtain the optimal bound by varying the weights that represent the same exponent vector. The full-support method appears in \cite[Section~4.2 and proof of Theorem~B]{BRS}. Remark~\ref{rem:support-extension} shows that extending the exponent vector by zero outside $S$ leaves the optimal constant unchanged. We nevertheless formulate the variational problem directly on $S$, using its quadratic derivatives and the relations among their ratios.

The full-support formula in \cite{BRS} uses ordinary coefficients and strictly Lorentzian polynomials. Passing to normalized coefficients multiplies the coefficient at $\alpha$ by $\alpha!$. Every full-support Lorentzian polynomial is a coefficientwise limit of strictly Lorentzian polynomials \cite[Theorem~2.25]{BH}. Since the ratios are continuous on positive coefficients, allowing all full-support Lorentzian polynomials leaves their suprema unchanged.

For a quadratic $M$-convex support $T\subseteq\Delta_3^2$, we form the vector of logarithms of the generator ratios in Corollary~\ref{cor:local-cones}, in the order given there. We write $C(T)$ for the set of these vectors as the quadratic varies over $\cL(T)$. For a vector $u$ with one weight for each generator, the support function \cite[Section~13]{Rockafellar} is $h_T(u)=\sup_{r\in C(T)}\langle u,r\rangle$. Thus $\exp h_T(u)$ is the optimal bound for their product with weights $u$. For full support, we write $h=h_{\Delta_3^2}$, so $C(\Delta_3^2)=C$. If $T$ has no generators, including $T=\emptyset$, we set $C(T)=\{0\}\subseteq\R^0$ and $h_T(0)=0$.

\Needspace{14\baselineskip}
For proper quadratic supports, the conditions in Lemma~\ref{lem:local-lorentzian-tests} give $C(T)$, and elementary optimization gives its support function.
\begin{lemma}
\label{lem:local-support-functions}
Up to permutation, the proper quadratic supports with generators have the following sets $C(T)$ and support functions for nonnegative weights.
$$
\begin{array}{c|c|c}
 T&C(T)&h_T\\ \hline
 \{11,12,22\}&(-\infty,0]&h_T(a)=0\\
 \{11,12,13,23\}&(-\infty,\log2]&h_T(a)=a\log2\\
 \{11,12,13,22,23\}&\{(r,s):e^r+e^s\leq2\}&h_T(a,b)\ \text{given below}
\end{array}
$$
In the last case, one has
$$
 h_T(a,b)=(a+b)\log2+a\log a+b\log b-(a+b)\log(a+b),
$$
where $0\log0=0$. For every support, $h_T$ is infinite if any weight is negative, and $h_T(0)=0$.
\end{lemma}

\begin{proof}
The sets $C(T)$ are given by Lemma~\ref{lem:local-lorentzian-tests}. The first two support functions are immediate. For the last support, we maximize $X^aY^b$ subject to $X+Y\leq2$. For $a,b>0$, the maximum occurs at $(2a/(a+b),2b/(a+b))$. The cases with a zero weight follow by continuity. If a weight is negative, we let its generator ratio tend to zero and keep the others equal to $1$. The supremum is then infinite. This also applies to full support.
\end{proof}

\begin{remark}
The function $h(a,b,c)$ is the logarithm of the sharp bound for $X^aY^bZ^c$. It is given by \cite[Theorem~C]{HHSW} on $a+b+c=1$ and extends to all nonnegative weights by positive homogeneity. As above, it is infinite outside the nonnegative orthant.
\end{remark}

The ratios of different quadratic derivatives share coefficients. To compute optimal constants subject to these relations, we use the following identity, where $\operatorname{ri}$ denotes relative interior \cite[Section~6]{Rockafellar}.
\begin{lemma}[\cite{Rockafellar}, Corollary~16.4.1]
\label{lem:support-duality}
Suppose that $D\subseteq\R^m$ is nonempty, closed, and convex and $L\subseteq\R^m$ is a linear subspace. If $L\cap\operatorname{ri}(D)\ne\varnothing$, then
$
 \sigma_{D\cap L}(\lambda)
 =\inf_{\eta\in L^\perp}\sigma_D(\lambda+\eta),
$
where $\sigma_D(\lambda)=\sup_{r\in D}\langle\lambda,r\rangle$.
\end{lemma}

\begin{proof}
By \cite[Theorem~16.4]{Rockafellar} and the relative-interior hypothesis, $\sigma_{D\cap L}(\lambda)$ is the infimum of $\sigma_D(\mu)+\sigma_L(\nu)$ over $\mu+\nu=\lambda$. Since $\sigma_L$ is zero on $L^\perp$ and infinite elsewhere, this is the stated formula.
\end{proof}

We fix a nonempty $M$-convex support $S\subseteq\Delta_3^d$ with $d\geq2$. In sums and products below, $\beta$ ranges over $\Delta_3^{d-2}$. For each nonzero quadratic derivative $\partial^\beta f$, we extend the generator vectors to $\R^S$ using $\iota_\beta$. We take these vectors as the rows of a matrix $M_S$, grouped by $\beta$ and ordered within each group as in $C(S_\beta)$. If a generator occurs in several derivatives, we retain one row for each occurrence. A derivative with no generators contributes no rows. For full support, each $\beta$ group has the rows $t_{\beta,1},t_{\beta,2},t_{\beta,3}$, in this order, as defined in Section~\ref{sec:preliminaries}. The matrix and its transpose then represent the maps in \cite[Definition~4.6 and (5.3)--(5.4)]{BRS}.

For $a_\alpha=\log c_\alpha$, the block $(M_Sa)_\beta$ consists of the logarithms of the generator ratios of $\partial^\beta f$. For a weight vector $\lambda$, we similarly write $\lambda_\beta$ for the weights assigned to this derivative. The resulting product of ratios has exponent $w=M_S^{\mathsf T}\lambda$.

The next lemma verifies the relative-interior hypothesis of Lemma~\ref{lem:support-duality}. We use the quadratic function from \cite[Lemma~5.2]{BRS}.
\begin{lemma}
\label{lem:strict-compatible-point}
Let $q=(q_\alpha)_{\alpha\in S}$ be given by $q_\alpha=-\|\alpha\|_2^2$, as in the proof of Theorem~\ref{thm:extreme-rays}. Then every block of $M_Sq$ lies in the relative interior of $C(S_\beta)$. In particular, $\operatorname{im}M_S$ meets the relative interior of $\prod_\beta C(S_\beta)$.
\end{lemma}

\begin{proof}
Whenever defined, the logarithm of $U$ equals $-4$, and those of $X,Y,Z$ equal $-2$. In particular, $U=e^{-4}<1$, $X+Y=2e^{-2}<2$, and, for full support, $XY=XZ=YZ=e^{-4}<1$ and $F(e^{-2},e^{-2},e^{-2})=2-3e^{-2}+e^{-6}>0$. The conditions in Lemma~\ref{lem:local-lorentzian-tests} are therefore strict whenever they impose a restriction.
\end{proof}

We next describe the logarithms of the generator ratios that arise from polynomials in $\cL(S)$. The full-support version appears in \cite[Lemma~5.1]{BRS}.
\begin{lemma}
\label{lem:general-feasible}
The set of logarithmic generator ratios is $M_S A(S)=\operatorname{im}M_S\cap\prod_\beta C(S_\beta)$.
\end{lemma}

\begin{proof}
Every polynomial in $\cL(S)$ gives a point in the intersection. Conversely, a point in the intersection has a preimage $a$ under $M_S$. We set $c_\alpha=e^{a_\alpha}$. By Lemma~\ref{lem:local-lorentzian-tests}, every nonzero quadratic derivative is Lorentzian. The polynomial has support $S$, so it belongs to $\cL(S)$.
\end{proof}

For nonnegative weights $\lambda$, we write $H_S(\lambda)=\log B_S(M_S^{\mathsf T}\lambda)$. Different weights can represent the same coefficient ratio and give different products of quadratic bounds. Taking the infimum over these weights gives the optimal bound, as in \cite[Lemma~4.9 and proof of Theorem~B]{BRS}.
\begin{theorem}
\label{thm:general-optimal}
Suppose that $S\subseteq\Delta_3^d$ is nonempty and $M$-convex, with $d\geq2$. For every nonnegative vector $\lambda$ with one weight for each row of $M_S$, one has
$$
 H_S(\lambda)=\inf_\mu\sum_{\beta\in\Delta_3^{d-2}}
 h_{S_\beta}(\mu_\beta),
$$
where the infimum is over nonnegative vectors $\mu$ satisfying $M_S^{\mathsf T}\mu=M_S^{\mathsf T}\lambda$.
\end{theorem}

\begin{proof}
Set $D=\prod_\beta C(S_\beta)$ and $L=\operatorname{im}M_S$. The set $D$ is closed and convex, with $\sigma_D(\mu)=\sum_\beta h_{S_\beta}(\mu_\beta)$. Lemmas~\ref{lem:general-feasible} and \ref{lem:strict-compatible-point} give $H_S(\lambda)=\sigma_{D\cap L}(\lambda)$ and $L\cap\operatorname{ri}(D)\ne\varnothing$. Lemma~\ref{lem:support-duality} therefore gives the required infimum, since $L^\perp=\ker M_S^{\mathsf T}$. By Lemma~\ref{lem:local-support-functions}, a negative weight makes the sum infinite, so it suffices to take $\mu\geq0$.
\end{proof}

For full support, we write $M_d=M_{\Delta_3^d}$ and $H_d=H_{\Delta_3^d}$.

Each quadratic derivative then contributes three triangular ratios, and its support function is $h$. The resulting formula agrees with \cite[Theorem~B, (5.2)]{BRS} for $d\geq3$ and with \cite[Theorem~C]{HHSW} for $d=2$.
\begin{corollary}
\label{cor:full-support-optimal}
For every nonnegative vector $\lambda=(\lambda_{\beta,i})$, one has $H_d(\lambda)=\inf_\mu\sum_\beta h(\mu_\beta)$, where the infimum is over nonnegative vectors $\mu$ satisfying $M_d^{\mathsf T}\mu=M_d^{\mathsf T}\lambda$.
\end{corollary}

\begin{proof}
The claim follows from Theorem~\ref{thm:general-optimal} with $S=\Delta_3^d$.
\end{proof}

\section{Relations between quadratic derivatives}\label{sec:compatibility-constraints}

For full support, we write $L_d=\operatorname{im}M_d$ for the space of vectors of logarithms of triangular ratios obtained by varying the positive coefficients. We determine the linear equations defining this space. Its orthogonal complement gives the relations among the rows of $M_d$, and hence the changes of weights allowed in Corollary~\ref{cor:full-support-optimal}.

Rescaling variables adds a linear function of $\alpha$ to the coefficient logarithms and leaves every triangular ratio unchanged. These are the only changes that preserve all triangular ratios, as in \cite[Sections~2.3 and~3]{BRS}.
\begin{lemma}
\label{lem:kernel-rho}
For $d\geq2$, $\ker M_d$ consists exactly of the vectors with $a_\alpha=b\cdot\alpha$ for some $b\in\R^3$. In particular, $\dim\ker M_d=3$.
\end{lemma}

\begin{proof}
Substitution shows that every vector with $a_\alpha=b\cdot\alpha$ lies in the kernel. Conversely, write $a_{i,j}=a_{(i,j,d-i-j)}$. For $i,j\geq0$ with $i+j\leq d-2$, the equation for $t_{\beta,3}$, where $\beta=(i,j,d-2-i-j)$, and its sum with the equation for $t_{\beta,1}$ give
$$
\begin{aligned}
 a_{i+1,j+1}-a_{i,j+1}&=a_{i+1,j}-a_{i,j},\\
 a_{i+2,j}-a_{i+1,j}&=a_{i+1,j}-a_{i,j}.
\end{aligned}
$$
Thus the horizontal first differences are constant within each row and equal in adjacent rows. By symmetry, the vertical first differences are also constant. Hence $a_{i,j}$ is affine in $i,j$. Since $\alpha_1+\alpha_2+\alpha_3=d$, this is equivalent to $a_\alpha=b\cdot\alpha$ for some $b\in\R^3$.
\end{proof}

\Needspace{4\baselineskip}
Rank-nullity now gives the number of independent relations.
\begin{corollary}
\label{cor:dual-dimension}
For $d\geq2$, one has $\dim L_d^\perp=(d-1)(d-2)$.
\end{corollary}

\begin{proof}
There are $3\binom d2$ triangular ratios and $\binom{d+2}{2}$ coefficients. Lemma~\ref{lem:kernel-rho} gives $\dim L_d^\perp=3\binom d2-\binom{d+2}{2}+3=(d-1)(d-2)$.
\end{proof}

The two relations for cubic polynomials in \cite[Definition~4.6 and Lemma~4.7]{BRS}, applied to derivatives of order $d-3$, form a basis of all relations in degree $d$.
\begin{theorem}
\label{thm:explicit-compatibility}
Suppose that $d\geq3$. A vector of logarithmic ratios $r=(r_{\beta,i})_{\beta\in\Delta_3^{d-2},\,i\in[3]}$ belongs to $L_d$ if and only if the following relations hold for every $\gamma\in\Delta_3^{d-3}$:
$$
\begin{aligned}
 r_{\gamma+e_3,1}+r_{\gamma+e_1,3}
 &=r_{\gamma+e_2,1}+r_{\gamma+e_2,3},\\
 r_{\gamma+e_3,2}+r_{\gamma+e_2,3}
 &=r_{\gamma+e_1,2}+r_{\gamma+e_1,3}.
\end{aligned}
$$
These $(d-1)(d-2)$ relations are linearly independent.
\end{theorem}

\begin{proof}
We identify each coordinate vector $\varepsilon_\alpha$ with the monomial $\mathbf x^\alpha$. Then $t_{\beta,i}$ corresponds to $\mathbf x^\beta(x_i-x_j)(x_i-x_k)$ for $\{i,j,k\}=\{1,2,3\}$. We represent the weights in a relation by a triple of polynomials. The coefficient of $\mathbf x^\beta$ in the $i$th polynomial is the weight of $t_{\beta,i}$. For each $\gamma\in\Delta_3^{d-3}$, the two relations in the statement are represented by $\mathbf x^\gamma(x_3-x_2,0,x_1-x_2),\quad \mathbf x^\gamma(0,x_3-x_1,x_2-x_1)$. Direct multiplication shows that each vector gives a relation among the three quadratic polynomials $(x_i-x_j)(x_i-x_k)$.

In a real linear relation among these vectors, the first component is $x_3-x_2$ times a linear combination of the monomials $\mathbf x^\gamma$. Multiplication by a nonzero polynomial is injective, so all coefficients of the first family vanish. The second component gives the same conclusion for the second family. The $(d-1)(d-2)$ relation vectors are therefore independent and form a basis of $L_d^\perp$ by Corollary~\ref{cor:dual-dimension}.
\end{proof}

This basis parametrizes all changes of weights allowed in Corollary~\ref{cor:full-support-optimal}.
\begin{corollary}
\label{cor:explicit-general-formula}
For $d\geq3$, every vector of weights representing the same exponent as $\lambda$ is obtained by adding to $\lambda$ a unique linear combination of the relation vectors in Theorem~\ref{thm:explicit-compatibility}. The infimum in Corollary~\ref{cor:full-support-optimal} therefore involves $(d-1)(d-2)$ real parameters, subject to nonnegativity of the resulting weights.
\end{corollary}

\begin{proof}
Two weight vectors have the same exponent exactly when their difference lies in $\ker M_d^{\mathsf T}=L_d^\perp$. The theorem gives a basis of this space.
\end{proof}

\section{Sharp bounds for ternary cubics}
\label{sec:cubic}

For ternary cubics with full support, the compatibility relations reduce the formula for optimal constants to an optimization over two real variables. We then determine the sharp constants for products of nonnegative powers of two distinct triangular ratios.

For $f\in\cL(3,3)$, we write $c_{abc}=c_{(a,b,c)}$ and arrange the nine ratios as a $3\times3$ matrix $T=(T_{ij})$, where $T_{ij}=T_{e_i,j}$. The row index specifies the derivative $\partial_i f$, and the column index specifies the diagonal entry of its Hessian, as in \cite[Definition~4.6]{BRS}. For example,
$$
 T_{11}=\frac{c_{111}c_{300}}{c_{210}c_{201}},\qquad
 T_{12}=\frac{c_{201}c_{120}}{c_{210}c_{111}}.
$$
A positive row $(X,Y,Z)$ comes from a Lorentzian quadratic exactly when its pairwise products are at most $1$ and $F(X,Y,Z)\geq0$, by Lemma~\ref{lem:local-lorentzian-tests}.

The relations imposed by shared coefficients, together with the conditions in Lemma~\ref{lem:local-lorentzian-tests}, characterize which matrices $T$ arise from a Lorentzian cubic \cite[Lemma~4.7]{BRS}.
\begin{lemma}
\label{lem:lifting}
Suppose that $T_{ij}>0$ for $1\leq i,j\leq3$.
\begin{enumerate}[label=\textup{(\roman*)}]
\item The matrix $T$ comes from positive cubic coefficients if and only if\\ $T_{13}T_{31}=T_{21}T_{23},\quad T_{23}T_{32}=T_{12}T_{13}$.
\item It comes from a polynomial in $\cL(3,3)$ if and only if these relations hold and every row satisfies the conditions in Lemma~\ref{lem:local-lorentzian-tests}.
\end{enumerate}
\end{lemma}

\begin{proof}
After taking logarithms, the two relations are the compatibility equations for $d=3$ in Theorem~\ref{thm:explicit-compatibility}. They therefore characterize the image of $M_3$. The row conditions then show that all three quadratic derivatives are Lorentzian.
\end{proof}

The two relations also imply $T_{12}T_{21}=T_{31}T_{32}$. We arrange a weight vector $\lambda$ as the $3\times3$ matrix $\Lambda=(\lambda_{ij})$, with each entry in the same position as its ratio in $T$. We write $H_3(\Lambda)$ for $H_3(\lambda)$ under this identification. For $u,v\in\R$, we define
$$
 E(u,v)=\begin{pmatrix}
 0&-v&u-v\\
 -u&0&v-u\\
 u&v&0
 \end{pmatrix}.
$$

\Needspace{9\baselineskip}
Adding $E(u,v)$ to the weights leaves the coefficient ratio unchanged. Every other choice of weights for the same ratio is obtained in this way, giving a two-variable formula for the optimal bound \cite[Theorem~4.11]{BRS}.
\begin{theorem}
\label{thm:cubic-optimal}
For every nonnegative $3\times3$ weight matrix $\Lambda$, $H_3(\Lambda)$ is the infimum of the sum of the values of $h$ at the three rows of $\Lambda+E(u,v)$. The infimum is taken over $u,v\in\R$ for which $\Lambda+E(u,v)$ is nonnegative.
\end{theorem}

\begin{proof}
The two compatibility relations in Lemma~\ref{lem:lifting} have coefficient matrices $E(1,0)$ and $E(0,1)$. They form a basis of $\ker(M_3^{\mathsf T})$, so the assertion follows from Corollary~\ref{cor:full-support-optimal}.
\end{proof}

Theorem~\ref{thm:cubic-optimal} reduces the sharp-bound problem to a two-parameter variational formula. For products of two distinct triangular ratios, we determine the constants directly from the quadratic inequalities and explicit compatible families. There are three cases according to the positions of the ratios in $T$ \cite[Theorem~4.11]{BRS}.
\begin{theorem}
\label{thm:cubic-two-generator}
For distinct ratios $T_{ij},T_{k\ell}$ and $a,b\geq0$, the optimal constant for $T_{ij}^aT_{k\ell}^b$ on $\cL(3,3)$ is as follows.
\begin{enumerate}[label=\textup{(\roman*)}]
\item If $i=k$, the constant is
$$
 2^{a+b}\frac{a^ab^b}{(a+b)^{a+b}}.
$$
\item If $(k,\ell)=(j,i)$ with $i\ne j$, the constant is $2^{|a-b|}$.
\item In all other cases, the constant is $2^{a+b}$.
\end{enumerate}
When either weight is zero, we use $0^0=1$. At $a=b=0$, the first expression is understood to be $1$.
\end{theorem}

\begin{proof}
To construct a cubic with a prescribed row of ratios, we may assume by permuting the variables that this is the first row. For a positive triple $(X,Y,Z)$ satisfying the conditions in Lemma~\ref{lem:local-lorentzian-tests} and sufficiently small $\varepsilon>0$, we choose the other rows to be $(\varepsilon Z,1,1)$ and $(\varepsilon,YZ,1)$. The resulting matrix satisfies the compatibility relations. The two new rows have pairwise products at most $1$ and values of $F$ equal to $0$ and $(1-\varepsilon)(1-YZ)$. Lemma~\ref{lem:lifting} therefore applies.

\medskip
\noindent\emph{Case (i): The two ratios lie in the same row.}
If both ratios lie in one row, the conditions in Lemma~\ref{lem:local-lorentzian-tests} give $X+Y\leq2$. Conversely, every positive pair with $X+Y<2$ extends to a row satisfying these criteria by choosing $Z>0$ sufficiently small. The case of two nonzero weights in \cite[Theorem~C]{HHSW} therefore gives the stated supremum. For $a,b>0$, this value is approached as $(X,Y)$ tends to $(2a/(a+b),2b/(a+b))$. The cases with a zero weight follow by taking limits.

\medskip
\noindent\emph{Case (ii): The two ratios occupy transposed positions.}
For a pair $T_{ij},T_{ji}$, we may assume by permuting the variables that the pair is $T_{12},T_{21}$. Compatibility gives $T_{12}T_{21}=T_{31}T_{32}\leq1$. This is also the normalized third inequality of \cite[Corollary~4.12]{BRS} after permuting the variables. Since each factor is at most $2$, for $a\geq b$ we have $T_{12}^aT_{21}^b =T_{12}^{a-b}(T_{12}T_{21})^b\leq2^{a-b}$.

For $0<t\leq1$, set $A_t=2/(1+t)$ and consider
$$
T(t)=
\begin{pmatrix}
t&A_t&t\\
A_t^{-1}&t&A_tt\\
1&1&t
\end{pmatrix}.
$$
The first and third rows satisfy $F=0$. For the second row, a direct calculation gives
$$
F(A_t^{-1},t,A_tt)
=\frac{(1-t)^2(2t+3)}{2(1+t)}\geq0.
$$
All pairwise products in each row are at most $1$. Moreover, the compatibility relations hold:
$$
T_{13}T_{31}=t=T_{21}T_{23},\qquad
T_{23}T_{32}=A_tt=T_{12}T_{13}.
$$
By Lemma~\ref{lem:lifting}, the matrix $T(t)$ comes from a polynomial in $\cL(3,3)$.

Since $T_{12}=A_t\to2$ and $T_{21}=A_t^{-1}\to1/2$ as $t$ tends to zero, the upper bound is sharp when $a\geq b$. The case $b\geq a$ follows by interchanging the first two variables.

\medskip
\noindent\emph{Case (iii): All remaining pairs.}
For $0<t\leq1$, let $A_t=2/(1+t)$ and write $v_i(t)$ for the row with $A_t$ in position $i$ and $t$ elsewhere. For the remaining pairs, the individual bounds give $2^{a+b}$. There are five types up to permutation of the variables: one with two diagonal entries, two with one diagonal entry, and two with neither. The following table gives the rows for one representative of each type.
$$
\begin{array}{c|ccc}
\text{Selected pair}&\text{Row }1&\text{Row }2&\text{Row }3\\ \hline
\{T_{11},T_{21}\}&v_1(t)&v_1(t)&v_1(t)\\
\{T_{11},T_{22}\}&v_1(t)&v_2(t)&v_3(t)\\
\{T_{12},T_{23}\}&v_2(t)&v_3(t)&v_1(t)\\
\{T_{12},T_{32}\}&v_2(t)&v_2(t)&v_2(t)\\
\{T_{11},T_{23}\}&v_1(t)&v_3(t)&(A_t,t^2/A_t,t)
\end{array}
$$
For $0<t\leq1$, the rows $v_i(t)$ satisfy the conditions in Lemma~\ref{lem:local-lorentzian-tests}, with $F(v_i(t))=0$. The row $(A_t,t^2/A_t,t)$ in the last line has pairwise products at most $1$ and satisfies $F(A_t,t^2/A_t,t) =\frac{t(1-t)^2(t+2)}{2(1+t)}\geq0$. Each matrix satisfies the compatibility relations in Lemma~\ref{lem:lifting}, and both selected ratios equal $A_t\to2$. This proves sharpness.
\end{proof}

\section{Applications}
\label{sec:applications}

We apply the preceding bounds to volume polynomials and to matroid basis polynomials with variables grouped by color.

\subsection{Mixed volumes}
\label{subsec:app-no-loss}

For a Lorentzian cubic with full support, we set $\Phi(f)=T_{11}T_{22}T_{33}$. We denote by $\mathcal V$ the set of volume polynomials with full support of the form $\operatorname{Vol}(x_1K_1+x_2K_2+x_3K_3)$, where $K_1,K_2,K_3$ are convex bodies in $\R^3$. We write $V(A,B,C)$ for the mixed volume of $A,B,C$ in $\R^3$. These polynomials are Lorentzian \cite[Theorem~4.1]{BH}, and their normalized coefficients are six times the corresponding mixed volumes \cite[Section~5.1]{Schneider}. The common factor $6$ cancels in each triangular ratio.

The bound $\Phi\leq8$ is also sharp for volume polynomials.
\begin{proposition}
\label{prop:app-no-loss}
$
 \sup_{G\in\mathcal V}\Phi(G)
 =\sup_{f\in\cL(3,3)}\Phi(f)=8.
$
\end{proposition}

\begin{proof}
The three bounds $T_{ii}\leq2$ give $\Phi\leq8$. For sharpness, we fix $0<\varepsilon\leq1/2$ and take the square pyramid
$$
 P_1=\operatorname{conv}
 \{0,e_1,\varepsilon e_2,\varepsilon e_3,
                 \varepsilon(e_2+e_3)\},
$$
where $P_2,P_3$ are obtained by cyclic coordinate permutations. For each $i$ and each $j\ne i$, the volume formula for a pyramid and the facet formula for mixed volume give
$$
 \operatorname{Vol}(P_i)=\frac{\varepsilon^2}{3},\qquad
 V(P_i,P_i,P_j)=\frac{\varepsilon+\varepsilon^2+\varepsilon^3}{6}.
$$

For the mixed-volume formula, symmetry reduces the calculation to $V(P_1,P_1,P_2)$. The three facets of $P_1$ lying in coordinate planes contribute zero, since the support function of $P_2$ vanishes at their outward normals.

The remaining two facets have outward normals proportional to $\nu_2=(\varepsilon,1,0)$ and $\nu_3=(\varepsilon,0,1)$. Each facet has area $\frac{\varepsilon}{2}\sqrt{1+\varepsilon^2}$, and both normals have length $\sqrt{1+\varepsilon^2}$. For $u\in\R^3$, set $h_{P_2}(u)=\max_{x\in P_2}\langle u,x\rangle$. Then $h_{P_2}(\nu_2)=1$ and $h_{P_2}(\nu_3)=\varepsilon+\varepsilon^2$.

The facet formula for mixed volumes \cite[Section~5.1]{Schneider} therefore gives
$$
\begin{aligned}
V(P_1,P_1,P_2)
&=\frac{\varepsilon}{6}
\bigl(h_{P_2}(\nu_2)+h_{P_2}(\nu_3)\bigr)\\
&=\frac{\varepsilon+\varepsilon^2+\varepsilon^3}{6}.
\end{aligned}
$$

To compute $V(P_1,P_2,P_3)$, we put $\Delta=\operatorname{conv}\{0,e_1,e_2,e_3\}$. Each $P_i$ contains the segment $[0,e_i]$ and is contained in $\Delta$. The equality $V([0,e_1],[0,e_2],[0,e_3])=V(\Delta,\Delta,\Delta)=1/6$ forces $V(P_1,P_2,P_3)=1/6$ by monotonicity \cite[Section~5.1]{Schneider}. The value of $\Phi$ on the volume polynomial of $P_1,P_2,P_3$ is $\frac{8}{(1+\varepsilon+\varepsilon^2)^6}$. This tends to $8$ as $\varepsilon$ tends to zero.
\end{proof}

We define
$$
 \Psi(f)=\frac{c_{201}c_{102}c_{030}}{c_{111}^3}
        =T_{12}T_{21}T_{22}T_{23}T_{32}.
$$

For this ratio $\Psi$, volume polynomials satisfy a strictly stronger bound.
\begin{proposition}
\label{prop:app-geometric-gap}
Let $C_*=\frac{32}{27}(13\sqrt{13}-46)\approx1.03368$. Then
$$
 \sup_{f\in\cL(3,3)}\Psi(f)=C_*,\qquad
 \sup_{G\in\mathcal V}\Psi(G)=1.
$$
\end{proposition}

\begin{proof}
The compatibility relations give $T_{13}T_{31}=T_{21}T_{23}=t^2$ for some $0<t\leq1$. The first and third rows, together with the arithmetic--geometric mean inequality, give $T_{12}T_{32}\leq(2-T_{13})(2-T_{31})\leq(2-t)^2$. The middle row gives $T_{22}\leq2/(1+t)$. For $t<1$, this follows from its determinant inequality and $T_{21}+T_{23}\geq2t$. At $t=1$, the conditions $T_{21}T_{23}=1$ and $F\geq0$ imply $T_{21}=T_{23}=1$, and hence $T_{22}\leq1$. Thus $\Psi(f)=t^2T_{22}T_{12}T_{32} \leq\frac{2t^2(2-t)^2}{1+t}$. The right-hand side is maximized at $t=(\sqrt{13}-1)/3$, with value $C_*$.

For sharpness, we fix this value of $t$ and consider the matrix
$$
 T=\begin{pmatrix}
 \varepsilon/2&2-t-\varepsilon&t\\
 t&2/(1+t)&t\\
 t&2-t-\varepsilon&\varepsilon/2
 \end{pmatrix}.
$$
For sufficiently small $\varepsilon>0$, all rows satisfy the conditions in Lemma~\ref{lem:local-lorentzian-tests}: $F=\varepsilon(1+t(2-t-\varepsilon))/2>0$ on the first and third rows, and $F=0$ on the middle row. The compatibility relations also hold, so by Lemma~\ref{lem:lifting} the matrix comes from a polynomial in $\cL(3,3)$. Its value of $\Psi$ tends to $C_*$ as $\varepsilon$ tends to zero.

For the volume polynomial $G(x,y,z)=\operatorname{Vol}(xA+yB+zC)$, the inequality $\Psi(G)\leq1$ is equivalent to
$$
 V(A,A,C)V(A,C,C)\operatorname{Vol}(B)\leq V(A,B,C)^3.
$$
Full support ensures that $A,B,C$ have nonempty interior. A result of Firey \cite{FireyMixedBodies}, discussed in \cite{LutwakMixedBodies}, gives a convex body $K$, called a mixed body of $A$ and $C$, such that $V(L,K,K)=V(L,A,C)$ for every convex body $L$. Taking $L=K$ and $L=C$, the Alexandrov--Fenchel inequality \cite[Section~7.3]{Schneider} gives
$$
 \operatorname{Vol}(K)^2
 =V(A,K,C)^2
 \geq V(A,A,C)V(K,K,C)
 =V(A,A,C)V(A,C,C).
$$
Taking $L=B$ gives $V(B,K,K)=V(A,B,C)$. Minkowski's first inequality \cite[Section~7.2]{Schneider} then gives
$$
 \operatorname{Vol}(B)V(A,A,C)V(A,C,C)
 \leq\operatorname{Vol}(B)\operatorname{Vol}(K)^2
 \leq V(A,B,C)^3.
$$
Equality holds when $A=B=C$.
\end{proof}

Thus a full-support Lorentzian cubic with $c_{201}c_{102}c_{030}>c_{111}^3$ cannot be a volume polynomial.

\subsection{Matroid basis counts}
\label{subsec:app-matroid}

We consider a rank-three matroid $M$ \cite{Oxley} and a partition $E_1\sqcup E_2\sqcup E_3$ of its ground set into three color classes. For nonnegative integers $a,b,c$ with $a+b+c=3$, we denote by $b_{abc}$ the number of bases containing $a,b,c$ elements of the respective classes. The polynomial $B(x_1,x_2,x_3)=\sum_{a+b+c=3}b_{abc}x_1^ax_2^bx_3^c$ is Lorentzian by \cite[Theorem~3.10]{BH} and preservation under identifying variables \cite[Theorem~2.10]{BH}.

Substituting $c_{abc}=a!b!c!\,b_{abc}$ into $\Phi\leq8$ gives
$$
 b_{111}^3b_{300}b_{030}b_{003}
 \leq\frac{64}{27}\,
 b_{210}b_{201}b_{120}b_{021}b_{102}b_{012}.
$$
This also holds when coefficients vanish, by positive linear substitution and passage to the limit \cite[Theorem~2.10]{BH}.

A result of Wagner \cite[Theorem~1.1]{WagnerRankThree} improves this bound for rank-three matroids. The resulting constant is sharp.
\begin{proposition}
\label{prop:app-matroid}
For every rank-three matroid and every three-coloring, one has
$$
 b_{111}^3b_{300}b_{030}b_{003}
 \leq b_{210}b_{201}b_{120}b_{021}b_{102}b_{012}.
$$
The constant $1$ is sharp.
\end{proposition}

\begin{proof}
Assign a variable $y_e$ to each element $e$ of $M$ and let $Z(\mathbf y)=\sum_A\prod_{e\in A}y_e$, where the sum runs over all bases $A$ of $M$. Subscripts denote partial derivatives. By \cite[Section~1 and Theorem~1.1]{WagnerRankThree}, the polynomial $Z$ satisfies $Z_eZ_f\geq ZZ_{ef}$ for distinct $e,f$ on the positive orthant. After substituting $y_e=x_i$ for $e\in E_i$, we sum these inequalities over $e\in E_2$ and $f\in E_3$ to obtain $B_2B_3\geq BB_{23}$. By continuity, this inequality holds at $(x_1,x_2,x_3)=(1,0,0)$, where it gives $b_{210}b_{201}\geq b_{300}b_{111}$. Multiplying the three cyclic versions proves the inequality.

For sharpness, we take a direct sum of three rank-one uniform matroids, each on $N+2$ elements for a positive integer $N$. In the three summands, the color classes have sizes $(N,1,1)$, $(1,N,1)$, and $(1,1,N)$, respectively \cite{Oxley}. The resulting polynomial is $B(x,y,z)=(Nx+y+z)(x+Ny+z)(x+y+Nz)$. Here $b_{300}=b_{030}=b_{003}=N$, $b_{111}=N^3+3N+2$, and the six remaining mixed coefficients equal $N^2+N+1$. The quotient of the left side of the asserted inequality by its right side tends to $1$ as $N$ tends to infinity.
\end{proof}

\noindent\textbf{Declaration of AI use.} The authors used ChatGPT to assist with calculations and revision. The authors take responsibility for the mathematical content of the manuscript.

\end{document}